\documentclass[a4paper,12pt]{amsart}
\usepackage[utf8]{inputenc}
\usepackage{amsmath}
\usepackage{amsfonts}
\usepackage{amsthm}
\usepackage{amssymb}
\usepackage[mathscr]{euscript}
\usepackage{xcolor}
\usepackage{enumerate, enumitem}
\usepackage{tikz}
\usepackage{mathtools}
\usepackage{comment}

\usepackage[draft=false,hidelinks]{hyperref}
\usepackage[noabbrev, capitalise, nameinlink]{cleveref} 

\usepackage{scalerel,stackengine}
\stackMath
\newcommand\reallywidehat[1]{%
\savestack{\tmpbox}{\stretchto{%
  \scaleto{%
    \scalerel*[\widthof{\ensuremath{#1}}]{\kern-.6pt\bigwedge\kern-.6pt}%
    {\rule[-\textheight/2]{1ex}{\textheight}}
  }{\textheight}%
}{0.5ex}}%
\stackon[1pt]{#1}{\tmpbox}%
}

\usepackage[
margin=1in,
marginpar=2cm,
includefoot,
footskip=30pt,
]{geometry}

\newtheorem{theorem}[equation]{Theorem}
\newtheorem{lemma}[equation]{Lemma}
\newtheorem{proposition}[equation]{Proposition}
\newtheorem{corollary}[equation]{Corollary}

\theoremstyle{definition}
\newtheorem{definition}[equation]{Definition}

\theoremstyle{remark}
\newtheorem{example}[equation]{Example}
\newtheorem{remark}[equation]{Remark}

\numberwithin{equation}{section}

\newcommand{\osf}{{\normalfont \textsf{X}}}

\newcommand{\lang}{\CL_{\osf}}
\newcommand{\algshift}{\CA_K(\osf)}
\newcommand{\ualgshift}{\TCA_K(\osf)}

\newcommand{\alf}{\mathscr{A}}

\newcommand{\N}{\mathbb{N}}

\newcommand{\id}{\mathrm{id}}
\newcommand{\CA}{\mathcal{A}}

\newcommand{\CL}{\mathcal{L}}

\newcommand{\TCA}{\widetilde{\CA}}
\newcommand{\TCB}{\mathcal{U}}

\newcommand{\HTCB}{\widehat{\TCB}}

\newcommand{\hsig}{\widehat{\sigma}}

\newcommand{\varphih}{\widehat{\varphi}}

\newcommand{\nn}{\mathbb{N}}
\newcommand{\zn}{\mathbb{Z}}

\newcommand{\scj}{\subseteq}

\newcommand{\dom}{\operatorname{dom}}
\newcommand{\supp}{\operatorname{supp}}

\title[Graded locally finite and just infinite Steinberg algebras]{Graded locally finite and just infinite Steinberg algebras}
\author[G. Boava]{Giuliano Boava}
\author[G.G. de Castro]{Gilles G. de Castro}
\author[D. Gonçalves]{Daniel Gonçalves}
\address[Giuliano Boava, Gilles G. de Castro and Daniel Gonçalves]{Departamento de Matem\'atica, Universidade Federal de Santa Catarina, 88040-970 Florian\'opolis SC, Brazil. }
\email{g.boava@ufsc.br \\ gilles.castro@ufsc.br \\ daemig@gmail.com}
\author[D.W. van Wyk]{Daniel W. van Wyk}
\address[Daniel van Wyk]{Department of Mathematics, Fairfield University, Fairfield, CT 06824 USA.}
\email{dvanwyk@fairfield.edu }



\begin{document}

\begin{abstract}
We study graded locally finite and graded just infinite Steinberg algebras of ample Hausdorff groupoids. For gradings by discrete groups induced by a cocycle, we characterize finite-dimensional homogeneous components in terms of finite cocycle fibres. Moreover, we obtain general criteria ensuring that, once one homogeneous component is infinite-dimensional, all of them are. Under suitable isotropy hypotheses, we show that for locally finite Steinberg algebras all irreducible representations act by finite rank operators. We also obtain necessary conditions for Steinberg algebras to be just infinite. In addition, we give a complete characterization of graded just infinite Steinberg algebras in terms of finite invariant reductions and identify conditions under which graded just infiniteness is equivalent to graded simplicity. 

We apply these results to Steinberg algebras of Deaconu--Renault groupoids and to subshift algebras. For Deaconu--Renault groupoids, we obtain dynamical criteria for graded just infiniteness. For subshift algebras, we show that there are no nontrivial finite-dimensional examples, we characterize when all homogeneous components are infinite-dimensional, and, in the finite-alphabet case, we prove that graded just infiniteness, graded simplicity, minimality of the associated groupoid, and hyper-cofinality of the subshift are equivalent.

\end{abstract}

\maketitle

\section{Introduction}

Steinberg algebras of ample Hausdorff groupoids provide a common algebraic framework for a wide range of
structures arising in algebra, dynamics, and operator algebras, including Leavitt path algebras,
inverse semigroup algebras, algebras of Boolean dynamical systems, and algebras associated to
Deaconu--Renault systems and subshifts. One of the main strengths of this approach is that many
ring-theoretic properties of a Steinberg algebra can be interpreted directly in terms of the orbit
structure, isotropy, and invariant subsets of the underlying groupoid. In this paper we develop this
interaction in the setting of graded local finiteness, graded just infiniteness, and graded simplicity.

Just infinite algebras and their graded counterparts have long played an important role in ring theory.
They appear in the study of rigidity and trichotomy phenomena for associative algebras
\cite{FarkasSmall2002,FarinaPendergrassRice2007}, in branch and self-similar constructions such as
Bartholdi's branch rings and tree enveloping rings \cite{Bartholdi2006}, and in operator-algebraic
settings through the theory of just-infinite \(\mathrm C^*\)-algebras \cite{GrigorchukMusatRordam2018}.
In graded situations, the notion of \emph{graded just infinite} is especially natural, since graded ideals
are often accessible through the underlying dynamics. A fundamental motivating example is the theorem of
Abrams, Aranda Pino, and Siles Molina showing that, for Leavitt path algebras with their standard
\(\mathbb Z\)-grading, graded just infiniteness is equivalent to graded simplicity \cite{APMlocal}.
One of our goals is to place this phenomenon into a general groupoid framework.

We begin by studying graded locally finite Steinberg algebras. A continuous cocycle
$c$ from a groupoid $G$ into a discrete group $H$ induces a grading $\{A_K(G)_h\}_{h\in H}$ on its Steinberg algebra $A_K(G)$. We characterize when the homogeneous components of
\(A_K(G)\) are finite-dimensional, purely in terms of the cocycle fibres. More precisely, we show that
\(A_K(G)\) is locally finite if and only if each fibre \(c^{-1}(h)\) is finite
(Theorem~\ref{Loc_finite_Steinb}). This leads to strong restrictions on the unit space and orbit
structure of \(G\), and yields representation theoretic consequences such as the algebra being CCR under additional isotropy hypotheses (Proposition~\ref{LF_CCR}). We also prove general transfer results showing that, under natural groupoid hypotheses, the infinite-dimensionality of one homogeneous component forces the infinite-dimensionality of all homogeneous components.

The second main theme of the paper is graded just infinite Steinberg algebras. Using the graded ideal
correspondence of Clark, Exel, and Pardo \cite{CEP}, we show that for ample Hausdorff groupoids whose
degree-zero fibre is strongly effective, a Steinberg algebra \(A_K(G)\) is graded just infinite
if and only if every proper closed invariant reduction \(G|_W\) is finite 
(Theorem~\ref{graded_justinf_steinb}). This gives a purely dynamical characterization of graded just
infiniteness. We then identify a broad topological condition, Property~(FIC), under which graded just
infiniteness is equivalent to graded simplicity (Theorem~\ref{thm:FIC-justinf-simple}). Along the way,
we also give a general criterion for when the identity fibre \(c^{-1}(e)\) is principal
(Proposition~\ref{prop:principal.fiber}), which is especially useful in the Deaconu--Renault setting.

We next turn to applications. For Deaconu--Renault systems we obtain criteria for finite-dimensionality,
local finiteness, the categorical and local Noetherian and Artinian properties, and related finiteness
conditions on the associated Steinberg algebras. We then specialize the general characterization of graded
just infiniteness to Deaconu--Renault groupoids, obtaining a dynamical criterion in terms of proper closed invariant subsets and their restricted dynamics
(Theorem~\ref{thm:graded-justinf-DR}). Under a natural topological hypothesis on points outside the domain of the local homeomorphism, we further prove that graded just infiniteness, graded simplicity, and minimality are equivalent for Deaconu--Renault groupoids (Theorem~\ref{DR_JI_Simple}).

The graph case is recovered as a special instance of this Deaconu--Renault picture. In particular, for the
boundary path groupoid of a row-finite graph, our general results recover the equivalence between graded
just infiniteness, graded simplicity, and cofinality, and in the locally finite case we also recover just
infiniteness \cite{APMlocal}. Thus the classical picture for Leavitt path algebras fits naturally inside
the broader groupoid-theoretic framework developed here.

We also apply our results to subshift algebras. In this setting the groupoid model is built from the Stone
space of the Boolean algebra generated by the sets \(C(\alpha,\beta)\), linking symbolic dynamics to
Steinberg algebras and, via the associated groupoid, to groupoid \(\mathrm C^*\)-algebras. This circle of
ideas has already proved fruitful in several directions: subshift algebras have been realized as partial
group algebras with relations and used to describe dynamical structure and simplicity
\cite{BoavaCastroGoncalvesVanWyk2025CCM}; their socle has been related to subshift conjugacy
\cite{GoncalvesRoyer2024Socle}; and on the operator-algebraic side, \(\mathrm C^*\)-algebras associated
to one-sided subshifts over arbitrary alphabets have been shown to encode substantial dynamical information
\cite{BCGWCStarSubshift}. In the present paper we prove that there are no nontrivial finite-dimensional
subshift algebras, characterize when all homogeneous components are infinite-dimensional, and show that a
subshift algebra is locally finite if and only if the underlying subshift is finite. For one-sided
subshifts over finite alphabets, we moreover show that graded just infiniteness, graded simplicity,
minimality of the associated groupoid, and hyper-cofinality of the subshift are all equivalent
(Corollary~\ref{cor:graded-justinf-subshift}). 

Taken together, these results show that graded just infiniteness is not an isolated algebraic phenomenon, but part of a broader dynamical picture that can be formulated at the level of groupoids. The groupoid
approach allows us to recover known results for graph algebras, extend them to Deaconu--Renault systems, and obtain new applications to subshift algebras. More broadly, it provides a unified setting in which graded finiteness properties, ideal structure, and dynamical invariants can be studied simultaneously.

\section{Background}

Throughout, $G$ denotes a locally compact Hausdorff groupoid. The set $G^{(0)}=\{\gamma\gamma^{-1}: \gamma\in G\}$, with the relative topology from $G$, denotes the  \emph{unit space} of $G$. Let $r,s: G\to G^{(0)}$ denote the range map, given by $r(\gamma)=\gamma\gamma^{-1}$, and the source map, given by $s(\gamma)=\gamma^{-1}\gamma$. A \emph{bisection} in $G$ is a set $U\subseteq G$ such that $r|_{U}$ and $s|_U$ are bijections onto $r(U)$ ans $s(U)$, respectively. We call $G$ an \emph{ample groupoid} if it has a basis of compact open bisections for its topology. We work exclusively with locally compact Hausdorff and ample groupoids in this work, which we refer to as ample groupoids.  
A set $U\subseteq G^{(0)}$ is \emph{invariant} if $r(s^{-1}(U))=U$. If $U\subseteq G^{(0)}$ is  invariant and either open or closed, then the restriction $G|_U:=s^{-1}(U)\cap r^{-1}(U)$ is an ample subgroupoid of $G$. The \emph{orbit} of $u\in G^{(0)}$ is the set $\mathcal{O}_u:=\{v\in G^{(0)} : \exists \gamma\in G \text{ such that } s(\gamma)=u, r(\gamma)=v\}$.  The relation $u\sim v \leftrightarrow v\in \mathcal O_u$ is an equivalence relation on $G^{(0)}$. The \emph{orbit space} is $G/G^{(0)} := \{\mathcal O_u: u\in G^{(0)}\}$, the set of all equivalence classes, with the quotient topology. A groupoid $G$ is transitive if $|G/G^{(0)}|=1$. The \emph{isotropy group} at  $u\in G^{(0)}$ is the group $G_u^u:=\{\gamma\in G: r(\gamma)=s(\gamma)=u\}$ with unit $u$.  The \emph{isotropy} bundle of $G$ is $\text{Iso}(G)\;:=\;\{g\in G:\ r(g)=s(g)\},$ and we write $\text{Iso}(G)^\circ$ for its interior.
We say that $G$ is \emph{effective} if $\text{Iso}(G)^\circ \;=\; G^{(0)}.$ We say that $G$ is \emph{strongly effective} if for every closed invariant
subset $C\subseteq G^{(0)}$, the restriction $G|_C$ is effective.

Let $G$ be an ample groupoid and let $K$ be a field. Then 
\[A_K(G) := \mathrm{span}_{K}\{1_U: U\subseteq G \text{ is a compact-open bisection} \}, \]
denotes the Steinberg algebra associated with $G$, where multiplication is convolution. On generators convolution is given by
\[ 1_{U}\ast 1_{V} := 1_{UV}, \]
where $UV=\{\gamma\eta \in G: \gamma\in U, \eta\in V\}$. When the field $K$ 
is clear from context or the results hold for all fields, we simply write $A(G)$. Let $c:G\to H$ be a continuous cocycle into a discrete group $H$ with identity $e$. Then $c$ induces an $H$-grading $A(G) = \oplus_{h\in H} A(G)_{h}$, where the homogeneous component at $h\in H$ is given by 
\[A(G)_h = \mathrm{span}_{K} \{1_U: U\subseteq c^{-1}(h) \text{ is a compact-open bisection} \}.\] 
In this paper we focus on graded Steinberg algebras. Conversely, gradings on Steinberg algebras are always induced by continuous cocycles into discrete groups (\cite{AraHazratLiSims2018}).

\section{Graded locally finite Steinberg algebras}
Fix an ample groupoid $G$, a field $K$, a discrete group $H$ with identity $e$, and a continuous cocycle $c:G\to H$. In this section we study and characterize when $A(G)$ is locally finite. We give a full characterization in Theorem~\ref{Loc_finite_Steinb}. Before that, however, we compare Steinberg algebras to the special case of Leavitt path algebras $L_K(E)$ of a row-finite directed graph $E$ with the canonical $\mathbb{Z}$-grading.

\begin{definition}
    A graded algebra $A=\bigoplus_{h\in H}$, where $H$ is a discrete group, is said to be \emph{graded locally finite} if each homogeneous component $A_h$ is a finite dimensional $K$-vector space.
\end{definition}

When there is no confusion about the grading involved, or it has been fixed, we drop the adjective ``graded'' and simply say $A$ is \emph{locally finite}.

To characterize locally finite Steinberg algebras, we need conditions for the homogeneous components $A(G)_h, h\in H$ to be finite dimensional $K$-vector spaces. The lemma below gives conditions that ensure whether a subset of the generators $\{1_{U}: U\subseteq c^{-1}(h)\}$ of $A(G)_h$ is linearly independent. 

\begin{lemma}\label{lem:LI}
    Let $\{U_j\}_{j=1}^n$ be nonempty, mutually disjoint, and compact-open subsets of $G$. Then $\{1_{U_j}\}_{j=1}^n$ is linearly independent.
\end{lemma}
    \begin{proof} 
        
        Suppose that $\{U_j\}_{j=1}^n$ are mutually disjoint and that $a_11_{U_1}+\cdots +a_n1_{U_n} = 0$. If $\gamma\in G$, then $\gamma$ belongs to at most one $U_j$. Choose any $\gamma\in U_1$. Then 
        \[0 = (a_11_{U_1}+\cdots + a_n1_{U_n})(\gamma) = a_1.\]
        It follows similarly, that $a_j =0$ for $j=2,3,\ldots, n$. Hence  $\{1_{U_j}\}_{j=1}^n$ is linearly independent.
    \end{proof}

With Lemma~\ref{lem:LI}, we can characterize when $A(G)_h$ is finite dimensional.

\begin{proposition}\label{finite_fibers}
   Let $h\in H$. Then $A(G)_h$ is finite dimensional if and only if $c^{-1}(h)$ is finite. In this case, $\{1_{\{\gamma\}}\}_{\gamma\in c^{-1}(h)}$ is a basis for $A(G)_h$, and $\dim (A(G)_h)=|c^{-1}(h)|$.
\end{proposition}
    \begin{proof}

        Suppose that $A(G)_h$ has finite dimension $n$ and that $|c^{-1}(h)|\geq n+1$. Let $\gamma_1,\ldots,\gamma_{n+1}\in c^{-1}(h)$ be distinct. By the Hausdorff property, there are mutually disjoint open sets $V_1,\ldots,V_{n+1}$ contained in $c^{-1}(h)$ such that $\gamma_i\in V_i$. Moreover, since $G$ is ample, each $\gamma_i$ has a compact-open bisection neighborhood inside $V_i$. Replacing each $V_i$ with this compact-open bisection neighborhood, we may assume that each $V_i$ is a compact-open. By Lemma~\ref{lem:LI}, $\{1_{V_i}\}_{i=1}^{n+1}$ is linearly independent, contradicting the fact that $A(G)_h$ has dimension $n$. This implies that $|c^{-1}(h)|\leq \dim (A(G)_h)$.
        
        Assume now that $c^{-1}(h)$ is finite. Then, since $c^{-1}(h)$ is Hausdorff with the relative topology, it is discrete and singletons are open. Thus, singletons are compact-open in $c^{-1}(h)$. Since $c^{-1}(h)$ is open in $G$, it follows that singleton are compact-open in $G$. Thus, $\{1_{\{\gamma\}}\}_{\gamma\in c^{-1}(h)}$ generates $A(G)_h$. Hence,  $A(G)_h$ is finite dimensional and $|c^{-1}(h)|\geq \dim (A(G)_h)$.
        
        We conclude that $|c^{-1}(h)|= \dim (A(G)_h)$, and that $\{1_{\{\gamma\}}\}_{\gamma\in c^{-1}(h)}$ is a basis for $A(G)_h$ when $A(G)_h$ is finite dimensional.
       
    \end{proof}

For Leavitt path algebras of finite graphs, with the usual $\mathbb{Z}$-grading, if one component is infinite dimensional, then all components are infinite dimensional by \cite[Theorem~1.4.]{APMlocal}. For general gradings on ample groupoids, the following example demonstrates that this need not hold.

\begin{example} \label{matrix.example}
     Consider the groupoid $G=\{1,2\}\times\{1,2\}$ with product given by
$(i,j)(j,k)=(i,k)$ for all $i,j,k\in\{1,2\}$. Define a cocycle
$c:G\to\zn$ by $c(i,j)=i-j$. Then $A(G)\cong M_2(K)$, and the induced
$\zn$-grading has a $2$-dimensional homogeneous component in degree $0$,
$1$-dimensional homogeneous components in degrees $-1$ and $1$, and zero dimensional
homogeneous components in all other degrees.
If we take $G^\infty=\bigsqcup_{n\in\nn}G$, then $A(G^\infty)$ is the infinite
direct sum of copies of $M_2(K)$. Moreover,
\[
c'((n,(i,j)))=c(i,j)
\qquad (n\in\nn,\ i,j\in\{1,2\})
\]
defines a cocycle on $G^\infty$. The grading induced on $A(G^\infty)$ by $c'$
has infinite-dimensional homogeneous components in degrees $0$, $1$, and $-1$,
while all other homogeneous components are zero dimensional.
     
\end{example}

While there is no Steinberg algebra version of \cite[Theorem~1.4.]{APMlocal} in general, we do have the following analog of \cite[Theorem~1.4.]{APMlocal}.

\begin{proposition}\label{prop:inf-dim-transfer-finite-family} Assume that $s(c^{-1}(h))=G^{(0)}$ for every $h\in \operatorname{im}(c)$. If $A(G)_h$ is infinite-dimensional for some $h \in \operatorname{im}(c)$, then $A(G)_g$ is infinite-dimensional for every $g \in \operatorname{im}(c)$.
\end{proposition}

\begin{proof}
    Observe that $r(c^{-1}(h))=s(c^{-1}(h^{-1}))$ for all $h\in\operatorname{im}(c)$. Hence, $r(c^{-1}(h))=G^{(0)}$ for every $h\in \operatorname{im}(c)$. Fix $h\in\operatorname{im}(c)$ and assume $A(G)_h$ is infinite-dimensional. By Proposition~\ref{finite_fibers}, $c^{-1}(h)$ is infinite, and by the same proposition, it is enough to show that $c^{-1}(g)$ is infinite for every $g\in\operatorname{im}(c)$. Let $g\in\operatorname{im}(c)$. In the case that $G^{(0)}$ is infinite, since $s(c^{-1}(g))=G^{(0)}$ we have that $c^{-1}(g)$ is infinite. Suppose now that $G^{(0)}$ is finite. Since $c^{-1}(h)$ is infinite, there exist $u,v\in G^{(0)}$ such that $G^v_u\cap c^{-1}(h)$ is infinite. Fix $\gamma_0\in G^v_u\cap c^{-1}(h)$. By hypothesis, there exists $\eta\in c^{-1}(g)$ such that $s(\eta)=u$. Thus, $c^{-1}(g)$ contains the following infinite set
    \[\{\eta\gamma_0^{-1}\gamma:\gamma\in G^v_u\cap c^{-1}(h)\},\]
    concluding the proof.
\end{proof}

The next theorem characterizes when $A(G)$ is locally finite and follows immediately from Proposition~\ref{finite_fibers}.

\begin{theorem}\label{Loc_finite_Steinb}
    Let $G$ be an ample Hausdorff groupoid and let $c:G\to H$ be a continuous cocycle into a discrete group $H$. Then $A(G)$ is locally finite for the $H$-grading induced by $c$ if and only if $|c^{-1}(h)|<\infty$ for all $h\in H$.
\end{theorem}

\begin{remark}\label{discrete_unitspace}
    If $A(G)$ is locally finite for an $H$-grading induced by $c$, then $G^{(0)}$ is finite and discrete, since $G^{(0)}\subseteq c^{-1}(e)$. Moreover, $G$ is also discrete. To see this, fix $x\in G$. Then there is an open bisection $U$ containing $x$ such that $r:U\to r(U)$ is homeomorphism (with $r(U)$ open in $G^{(0)}$). Since $G^{(0)}$ is discrete, $\{r(x)\}$ is open in $r(U)$. Then $x=r|_{U}^{-1}(r(x))$ is open in $U$, and hence in $G$. 
\end{remark}

Suppose $u\in G^{(0)}$. Note that if $v\in \mathcal{O}_u$, then $G_u^u \cong G_v^v$. Thus, up to isomorphism, each orbit has one isotropy group.   

\begin{corollary} \label{direct_matrices}
    Suppose that $A(G)$ is locally finite with respect to a grading induced by a cocycle $c:G\to H$. Then, $G$ has finitely many orbits $\mathcal{O}_1, \ldots, \mathcal{O}_k$,  $|\mathcal{O}_i|=n_i<\infty$ for $i=1,\ldots,k$, and
    \[ A(G) = \bigoplus_{i=1}^k M_{n_i}(A(G_{u_i}^{u_i})) \]
    where $G_{u_i}^{u_i}$ is the isotropy group of $u_i\in \mathcal{O}_i$.
\end{corollary}
    \begin{proof}
        If $A(G)$ is locally finite, then $G^{(0)}\subseteq c^{-1}(e)$ is finite by Theorem~\ref{Loc_finite_Steinb}. The conclusion now follows from \cite[Proposition~3.1]{Steinberg_ChainConditions_2018}.
    \end{proof}

We say that $A(G)$ is \emph{CCR} if all its irreducible representations over $K$ act by finite rank operators \cite[Proposition~2.7]{CSW_20}. The following result describes a class of locally finite Steinberg algebras that are CCR. This class includes many combinatorial algebras such as Leavitt path algebras of graphs \cite{AbrAraMol} and of labeled graphs \cite{BCGW21}, algebras of subshifts \cite{BCGWSubshift}, and, more generally, Steinberg algebras of Deaconu-Renault groupoids. For this result we need extra assumptions on $G$ to ensure that the cardinality of the field $K$ is large enough. 
\begin{proposition}\label{LF_CCR}
    Suppose that $G$ is a second countable Hausdorff ample groupoid with virtually abelian isotropy, let $H$ be a discrete group, and let $K$ be an algebraically closed field of cardinality greater than $\dim A_K(G)$. If $A_K(G)$ is locally finite for an $H$-grading, then $A_K(G)$ is a CCR algebra. In particular, if $A_\mathbb{C}(G)$ is locally finite then it is CCR.
\end{proposition}
\begin{proof}
    Since $G_u^u$ is virtually abelian for all $u\in G^{(0)}$, $K$ is an algebraically closed field and all isotropy groups have cardinality less than $K$, it follows from \cite[Corollary~4.2]{CSW_20} that $A(G_u^u)\cong KG_u^u$ is CCR for every $u\in G^{(0)}$. 

    If $A(G)$ is locally finite, then $G^{(0)}\subseteq c^{-1}(e)$ is finite and discrete. Hence, all orbits are closed. Thus, $A(G)$ is CCR by \cite[Theorem~6.1]{CSW_20}. 
\end{proof}

\section{Just infiniteness and graded just infiniteness for Steinberg algebras}

Fix an ample groupoid $G$, a field $K$, and let $c:G\to H$ be a continuous cocycle into a discrete group $H$ with identity $e$. In this section we study when $A(G)$ is just infinite and graded just infinite. We start with an elementary observation that  will be useful later on, followed by definitions pertinent to this section. 

\begin{lemma}\label{finite_dim_Steinb}
    We have that $\dim{A(G)}< \infty$ iff $|G|<\infty$. 
\end{lemma}
    \begin{proof}
        Suppose $\dim{A(G)}< \infty$.  Then consider the trivial grading and Proposition~\ref{finite_fibers}. 

        Conversely, if $G$ is Hausdorff and $|G|<\infty$, then $G$ is discrete. Hence $\{1_{\{\gamma\}}\}_{\gamma\in G}$ forms a basis for $A(G)$.
    \end{proof}

\begin{definition}
    If $A=\bigoplus_{h}A_h$ is an $H$-graded algebra, then an ideal $I\subseteq A$ is a \emph{graded ideal} if $I=\bigoplus_{h}(A_h\cap I)$.
\end{definition}

\begin{definition}
A $K$-algebra $A$ is \emph{just infinite} if $\dim{A}=\infty$, and $\dim{A/I}<\infty$ for every non-zero ideal $I\subseteq A$. If $A$ is a graded algebra, then $A$ is \emph{graded just infinite} if $\dim{A}=\infty$, and  $\dim{A/J}<\infty$ for every non-zero graded ideal $J\subseteq A$.
\end{definition}

\begin{remark}
    Example 2.1 in \cite{AbrAraMol} shows that a graded just infinite Steinberg algebras need not be just infinite.   
\end{remark}

Let $U\subseteq G^{(0)}$ be an open invariant subset. Then $A(G|_U)$ is an ideal in $A(G)$ by \cite[Proposition~2.9]{CEP}. Let $W=G^{(0)}\setminus U$. Then $W$ is a closed invariant subset of $G^{(0)}$, and thus $G|_W$ is an ample Hausdorff groupoid. Hence we have the short exact sequence
\[ 0\to A(G|_U) \overset{\iota}{\to} A(G) \overset{\pi}{\to} A(G|_W)\to 0.\] 
This fact is likely well-known, but we provide an outline for completeness. If $U\subseteq G^{(0)}$ is an open invariant subset, then we have an inclusion $A(G|_U) \hookrightarrow A(G)$ and, moreover, $A(G|_U) = A(G)A(U)=A(U)A(G)$ is an ideal in $A(G)$ generated by $A(U)$ \cite[Proposition~2.9]{CEP}. The map $\pi : A(G) \to A(G|_W)$ is given by restriction. To see that it is onto, suppose $1_{\hat{V}}\in A(G|_W)$ with $\hat{V}\subseteq G|_W$ a compact open bisection. Then, $\hat{V} = V\cap G|_W$ for some compact open set $V\in G$. Thus, $1_{\hat{V}} = 1_V|_{G|_W}=\pi(1_V)$. To see that $\ker{\pi}=A(G|_U)$, notice that $G=G|_U\sqcup G|_W$. Thus, if $f\in A(G)$ is nonzero and $f|_{G|_W}\equiv 0$, then  $\supp{f}\subseteq G|_U$. Hence, $f\in A(G|_U)$. If $g\in A(G|_U)$, then $\supp{g}\subseteq G|_U$, and so $g(\gamma)=0$ for all $\gamma\in G|_W$. Hence, we have that 
\[ A(G)/A(G|_U) \cong A(G|_W).\]

We now introduce a class of ample groupoids that shares many features with the Deaconu--Renault groupoids arising from directed graphs.

\begin{definition}
\label{def:FIC}
We say that $G$ satisfies \emph{Property (FIC)}--finite, invariant, and closed--if every nonempty, invariant subset $W\subseteq G^{(0)}$ with $|G|_W|<\infty$ is clopen in $G^{(0)}$.
\end{definition}

\begin{example}\label{ex:LF.implies.FIC}
    If $G$ is a groupoid such that $A(G)$ is locally finite for some $H$ grading, then, by Remark~\ref{discrete_unitspace}, $G^{(0)}$ is finite, so every $W\subseteq G^{(0)}$ is clopen. Hence, $G$ satisfies Property (FIC).
\end{example}

The following lemma shows that when $G^{(0)}$ can be decomposed into clopen invariant subsets, then $A(G)$ may be decomposed into a direct sum of graded ideals. 

\begin{lemma}[Clopen invariant decomposition]
\label{lem:clopen-decomp}
Let $G$ be an ample, Hausdorff, \'etale groupoid and let $U,W\subseteq G^{(0)}$ be
disjoint clopen invariant subsets with $G^{(0)}=U \sqcup W$.
Then $A(G|_U)$ and $A(G|_W)$ are graded ideals of $A(G)$ (for any grading coming from
a cocycle), and the map
\[
A(G|_U)\oplus A(G|_W)\ \longrightarrow\ A(G),\qquad (f,g)\mapsto f+g
\]
is a graded isomorphism. In particular,
\[
A(G)\ \cong\ A(G|_U)\ \oplus\ A(G|_W)
\]
as graded algebras.
\end{lemma}

\begin{proof}
Since $U$ and $W$ are invariant, $G$ decomposes as a disjoint union of open subgroupoids $G = (G|_U)\ \sqcup\ (G|_W),$ 
and both $G|_U$ and $G|_W$ are open and closed in $G$ because $U$ and $W$ are clopen.

Any $f\in A(G)$ is a compactly supported, locally constant $K$--valued function on $G$.
Because $G|_U$ and $G|_W$ are clopen in $G$, the restrictions $f|_{G|_U}$ and $f|_{G|_W}$
belong to $A(G|_U)$ and $A(G|_W)$, respectively, with disjoint supports, and
\[
f = (f|_{G|_U}) + (f|_{G|_W}).
\]
Thus the sum map is surjective.
Injectivity follows because $G|_U$ and $G|_W$ are disjoint, so if $f+g=0$ with
$f\in A(G|_U)$ and $g\in A(G|_W)$, then $f=0=g$.
Finally, if $A(G)$ is graded via a cocycle $c:G\to H$, then $G|_U$ and $G|_W$ are
$H$--graded subgroupoids and the restriction maps preserve degrees, so the isomorphism
is graded.
\end{proof}

We next record necessary conditions for $A(G)$ to be just infinite, according to the assumptions imposed on $G$. 

\begin{proposition}\label{justinf_steinb}
    Suppose that $G$ is an ample Hausdorff groupoid. If $A(G)$ is just infinite, then 
    \begin{enumerate}
        \item $|G|_W|<\infty$ for every nonempty closed invariant and  proper subset $W$ of $G^{(0)}$; 
        \item if $u\in G^{(0)}$ and $\overline{\mathcal{O}_u}\neq G^{(0)}$, then $|G_u^u|<\infty$; 
        \item every orbit is either dense or finite;
    \end{enumerate}   
    If, in addition, 
    \begin{enumerate}[start=4]
        \item  $G^{(0)}$ is finite, then $G$ is transitive.
        \item  $G$ is second countable, then $G$ has a dense orbit; 
        \item  $G$ satisfies condition (FIC), then $G$ is minimal.
    \end{enumerate}   
\end{proposition}

    \begin{proof}
        We note that $G$ is infinite by Lemma~\ref{finite_dim_Steinb}.
    
        (1): Suppose that $W\subseteq G^{(0)}$ is a nonempty proper closed invariant set. Then $U=G^{(0)}\backslash W$ is a nonempty open invariant set. Hence, $A(G|_U)$ is a nontrivial ideal in $A(G)$ such that $A(G)/A(G|_U) \cong A(G|_W)$. By assumption, $\dim{A(G|_W)}<\infty$. Hence, $|G|_W|<\infty$ by Lemma~\ref{finite_dim_Steinb}.

        (2) and (3): Let $u\in G^{(0)}$ and suppose that $ \overline{\mathcal{O}_u}\neq G^{(0)}$. Then the orbit closure $ \overline{\mathcal{O}_u}$ is a nonempty, proper and closed invariant subset of $G^{(0)}$. Hence, $|G|_{\overline{\mathcal{O}_u}}|<\infty$, by Item~(1). Since $G_u^u, \mathcal{O}_u$, and $\overline{\mathcal{O}_u}$ are each a subset of $G|_{\overline{\mathcal{O}_u}}$, it follows that $|G_u^u|<\infty$, $|\overline{\mathcal{O}_u}|<\infty$, and $|\mathcal{O}_u|<\infty$.

        (4): Assume that $G^{(0)}$ is finite. If $G$ is not transitive, then $G^{(0)}$ decomposes into finitely many nonempty proper closed invariant subsets. By \textnormal{(1)}, each corresponding reduction is finite, and therefore $G$ is a finite union of finite sets. This contradicts the assumption that $A(G)$ is infinite-dimensional.

        (5) Suppose, towards a contradiction, that $G$ has no dense orbits.   

        Let $U\subseteq G^{(0)}$ be a nonempty proper open invariant subset, and set $        W:=G^{(0)}\setminus U.$
        Then $A_K(G|_U)$ is a nonzero proper ideal of $A_K(G)$, so, since $A_K(G)$ is just infinite,
        the quotient
        \[
        A_K(G)/A_K(G|_U)\cong A_K(G|_W)
        \]
        is finite-dimensional. By Proposition~\ref{finite_fibers} applied to the trivial cocycle,
        it follows that the groupoid $G|_W$ is finite. In particular, $W$ is finite. Hence every nonempty proper open invariant subset of $G^{(0)}$ has finite complement.

        Now fix $u\in G^{(0)}$. Since $G$ has no dense orbits, the orbit $\mathcal O_u$ is not dense in $G^{(0)}$, so its closure $\overline{\mathcal O_u}$ is a proper closed invariant subset of $G^{(0)}$.
        Therefore its complement
        \[
        U_u:=G^{(0)}\setminus \overline{\mathcal O_u}
        \]
        is a nonempty proper open invariant subset. By the previous paragraph,
        $\overline{\mathcal O_u}$ is finite. Since $\mathcal O_u$ is dense in its closure and the closure is finite
        Hausdorff, we must have $        \mathcal O_u=\overline{\mathcal O_u}.$
        Thus every orbit is finite and closed.
        
        Let
        \[
        q:G^{(0)}\to G/G^{(0)}
        \]
        be the quotient map onto the orbit space. Since $G$ is étale, $q$ is open, and since all orbits are closed,
        the orbit space $G/G^{(0)}$ is a $T_1$ space. Also, as an open continuous image of the second countable space
        $G^{(0)}$, the orbit space is second countable. Finally, because $G^{(0)}$ is locally compact Hausdorff,
        it is a Baire space, and since $q$ is open and surjective, the orbit space $G/G^{(0)}$ is also Baire.
        
        We claim that every nonempty proper open subset of $G/G^{(0)}$ has finite complement.
        Indeed, if $V\subseteq G/G^{(0)}$ is nonempty proper open, then $q^{-1}(V)$ is a nonempty proper open
        invariant subset of $G^{(0)}$, so its complement is finite. Since the fibres of $q$ are precisely the
        orbits, it follows that $
        (G/G^{(0)})\setminus V$
        is finite.  We have therefore shown that the orbit space $Y:=G/G^{(0)}$ is an infinite second countable \(T_1\) Baire  space such that every nonempty proper open subset of \(Y\) has finite complement. We now show that this is
        impossible.
        
        Fix \(y_0\in Y\), and let \(\{B_n\}_{n\in\mathbb N}\) be a countable basis for \(Y\). For each \(n\) such that
        \(y_0\in B_n\), the complement \(Y\setminus B_n\) is finite. We claim that
        \[
        Y\setminus\{y_0\}\subseteq \bigcup_{\{n:\,y_0\in B_n\}} (Y\setminus B_n).
        \]
        Indeed, if \(y\neq y_0\), then \(Y\setminus\{y\}\) is a nonempty open set containing \(y_0\), so some basis
        element \(B_n\) satisfies
        $
        y_0\in B_n\subseteq Y\setminus\{y\}.$
        Hence \(y\in Y\setminus B_n\), proving the claim. Since the right-hand side is a countable union of finite sets,
        it follows that \(Y\setminus\{y_0\}\) is countable. Thus \(Y\) itself is countable.
        
        Enumerate \(Y=\{y_0,y_1,y_2,\dots\}\). For each \(m\in\mathbb N\), the set
        $
        U_m:=Y\setminus\{y_m\}$
        is open and dense (it is open because \(Y\) is \(T_1\), and it is dense because \(\{y_m\}\) is not open,
        otherwise \(Y\) would be finite). But $
        \bigcap_{m\in\mathbb N} U_m=\emptyset,$
        contradicting that \(Y\) is a Baire space.
        
        This contradiction shows that no such groupoid \(G\) can exist.

        (6) Toward a contradiction, suppose that $G$ has a non-dense orbit $\mathcal{O}_u$, for some $u\in G^{(0)}$. Then $|\mathcal{O}_u|<\infty$, by Item~(3); thus $\mathcal{O}_u=\overline{\mathcal O_u}$, and so $|G|_{\mathcal O_u}|<\infty$ by Item~(1). Thus condition (FIC) implies that $\mathcal O_u$ is clopen. Let $\mathcal O_u^c = G^{(0)} \setminus \mathcal O_u$, which is also clopen and decomposes the unit space into $G^{(0)}=\mathcal{O}_u\sqcup \mathcal{O}_u^c$.  

        By Lemma~\ref{lem:clopen-decomp}, 
        \begin{eqnarray}\label{eq:direct_sum_decomp}
            A(G) = A(G|_{\mathcal{O}_u}) \oplus A(G|_{\mathcal{O}_u^c}),
        \end{eqnarray}
        and moreover, each summand is a graded ideal. Since $A(G)$ is just infinite, both 
        \[A(G|_{\mathcal{O}_u}) = A(G) / A(G|_{\mathcal{O}_u^c}), \text{ and } 
            A(G|_{\mathcal{O}_u^c}) = A(G) / A(G|_{\mathcal{O}_u})\]
        are finite dimensional. But then Equation~(\ref{eq:direct_sum_decomp}) implies that $A(G)$ is finite dimensional, a contradiction. Hence, all orbits are dense in $G^{(0)}$.

    \end{proof}

We are unaware of a characterization of all ideals in a Steinberg algebra in terms of the groupoid's topological structure, and thus completely characterizing just infinite Steinberg algebras seems out of reach at the moment. The following theorem will be crucial to characterize graded just infinite Steinberg algebras: 

\begin{theorem}\cite[Theorem~5.3]{CEP}\label{graded_ideals}
    (Graded Ideal Correspondence) If $c^{-1}(e)$ is strongly effective, then the correspondence 
    \[U \to A(G|_U)\]
    is an isomorphism from the lattice of open invariant sets on $G^{(0)}$ onto the lattice of graded ideals in $A(G)$.
\end{theorem}

Given Theorem~\ref{graded_ideals}, we can completely characterize graded just infinite Steinberg algebras for which $c^{-1}(e)\subseteq G$ is strongly effective.

\begin{theorem}\label{graded_justinf_steinb}
    Suppose that $G$ is an infinite ample Hausdorff groupoid and $c:G\to H$ is a continuous cocycle such that $c^{-1}(e)$ is strongly effective. Then, $A(G)$ is graded just infinite if and only if $|G|_W|<\infty$ for every proper, closed and invariant subset $W$ of $G^{(0)}$.   
\end{theorem}
    \begin{proof}
      We start by observing that $A(G)$ is infinite dimensional by Lemma \ref{finite_dim_Steinb}.
    
      If $A(G)$ is graded just infinite, then applying Theorem \ref{graded_ideals} and the same argument used to prove Proposition~\ref{justinf_steinb}(1), we obtain that $|G|_W|<\infty$ for every nonempty proper closed invariant subset $W\subseteq G^{(0)}$. 

      Let $U$ be a nonempty,  proper, open invariant subset of $G^{(0)}$ and let $W=G^{(0)}\setminus U$. Then $W$ is a proper closed invariant subset of $G^{(0)}$. Thus, $|G|_W|<\infty$ by hypothesis. Then $A(G)/A(G|_U)\cong {A(G|_W)}$, which is finite dimensional by Lemma~\ref{finite_dim_Steinb}. Since $c^{-1}(e)$ is strongly effective, it follows from Theorem~\ref{graded_ideals} that all nonzero graded ideals of $A(G)$ are of the form $A(G|_U)$ for a nonempty open invariant subset $U\subseteq G^{(0)}$. Hence, $A(G)$ is graded just infinite.
    \end{proof}

\begin{corollary}
    Suppose that $G$ is an infinite, ample, Hausdorff, and strongly effective  groupoid. Then $A(G)$ is just infinite if and only if  $|G|_W|<\infty$ for every proper, closed invariant subset $W$ of $G^{(0)}$. 
\end{corollary}
    \begin{proof}
        Consider the trivial grading and apply Theorem~\ref{graded_justinf_steinb}.
    \end{proof}

Given Theorem~\ref{graded_justinf_steinb}, it is useful to have concrete criteria ensuring that
$c^{-1}(e)$ is strongly effective. In many important examples, however, one has the stronger
property that $c^{-1}(e)$ is principal; see, for instance, \cite[Lemma~7.3]{CEP}.
Motivated by this, we next characterize when $c^{-1}(e)$ is principal for locally compact Hausdorff groupoids.

\begin{proposition}\label{prop:principal.fiber}
Let $\mathcal{G}$ be a locally compact Hausdorff groupoid with unit space $\mathcal{G}^{(0)}$, let $\Gamma$ be a group, and let $c\colon \mathcal{G} \to \Gamma$ be a continuous cocycle. Then $c$ is injective on each isotropy group if and only if the identity fibre $\mathcal{G}_e := c^{-1}(e)$ is a principal groupoid. In particular, $c^{-1}(e)$ is strongly effective.
\end{proposition}

\begin{proof}
For each $u \in \mathcal{G}^{(0)}$, the restriction $c|_{\mathcal{G}_u^u} \colon \mathcal{G}_u^u \to \Gamma$ is a group homomorphism, and its kernel is
\[
\ker\bigl(c|_{\mathcal{G}_u^u}\bigr) = \{\gamma \in \mathcal{G} : r(\gamma) = s(\gamma) = u,\; c(\gamma) = e\} = (\mathcal{G}_e)_u^u.
\]
Therefore
\[
c|_{\mathcal{G}_u^u} \text{ is injective} \iff \ker\bigl(c|_{\mathcal{G}_u^u}\bigr) = \{u\} \iff (\mathcal{G}_e)_u^u = \{u\}.
\]
Since this holds for every $u \in \mathcal{G}^{(0)}$, we conclude that $c$ is injective on isotropy if and only if $\mathcal{G}_e$ is principal.
\end{proof}

Suppose further that $G$ is second countable. It is then natural to ask when a just infinite Steinberg algebra is CCR. The next result shows that just infiniteness imposes strong restrictions on the dimensions of irreducible representations.
 
 \begin{proposition}
Let $G$ be an infinite second countable Hausdorff ample groupoid, and let $K$ be an algebraically closed field with $|K|>\dim A_K(G)$. If $A_K(G)$ is just infinite, then $A_K(G)$ is CCR if and only if $G$ is transitive and all isotropy groups are CCR. In particular, if all isotropy groups are virtually abelian, then $A_{\mathbb C}(G)$ is CCR if and only if $G$ is transitive.
\end{proposition}
\begin{proof}
        
 By \cite[Theorem~6.1]{CSW_20}, the algebra $A_K(G)$ is CCR if and only if all orbits are closed and all isotropy groups are CCR. Thus, in order to prove the proposition, it remains to determine when all orbits are closed under the additional assumption that all isotropy groups are CCR.

Since $A_K(G)$ is just infinite, Proposition~\ref{justinf_steinb}(3) shows that every orbit is either finite or dense. Moreover, by Proposition~\ref{justinf_steinb}(5), $G$ has at least one dense orbit, say $\mathcal O_u$ for some $u\in G^{(0)}$. Finite orbits are closed, so it remains only to consider dense orbits. A dense orbit is closed if and only if it coincides with $G^{(0)}$, that is, if and only if $G$ is transitive.

Now assume that $G$ is not transitive. Then the dense orbit $\mathcal O_u$ is a proper dense subset of $G^{(0)}$, and hence is not closed. By \cite[Theorem~6.1]{CSW_20}, it follows that $A_K(G)$ is not CCR.

Conversely, suppose that $A_K(G)$ is CCR. Then all orbits are closed, so the dense orbit $\mathcal O_u$ must be equal to $G^{(0)}$. Therefore $G$ is transitive.

Finally, if $G$ is transitive and all isotropy groups are CCR, then the unique orbit is $G^{(0)}$, which is closed. Hence \cite[Theorem~6.1]{CSW_20} implies that $A_K(G)$ is CCR.\footnote{Note that this implication does not require $A_K(G)$ to be just infinite.}

If all isotropy groups are virtually abelian, then the final statement follows from the preceding equivalence together with \cite[Proposition~2.71]{CSW_20}.
\end{proof}

\section{Graded just infiniteness versus graded simplicity for Steinberg algebras}

In this section we consider ample groupoids whose Steinberg algebras are graded just infinite with respect to a given cocycle. In \cite[Theorem~2.4]{APMlocal} the authors show that for infinite dimensional Leavitt path algebras with the standard $\mathbb{Z}$-grading, being graded just infinite is equivalent to being graded simple. We aim to find a groupoid version of \cite[Theorem~2.4]{APMlocal} that generalizes it. 

Fix an ample groupoid $G$, a field $K$, and let $c:G\to H$ be a continuous cocycle into a discrete group $H$ with identity $e$. 

We start by observing that we cannot obtain an analog of \cite[Theorem~2.4]{APMlocal} for general Steinberg algebras, as the following example shows.

\begin{example}\label{No_LPA_analog}
Consider the groupoid $G=(\nn\times\nn)\cup\{\infty\}$, where the unit space is $\nn\cup\{\infty\}$, the one-point compactification of $\nn$, and every pair $(n,m)\in\nn\times\nn$ is an isolated point. The operation is given by $(n,m)(m,k)=(n,k)$ for every $n,m,k\in\nn$. Consider the cocycle $c:G\to\zn$, such that $c(n,m)=n-m$ for $n,m\in\nn$ and $c(\infty)=0$. Then $G$ satisfies the hypotheses of Theorem \ref{graded_justinf_steinb}, so $A(G)$ is graded just infinite, since the only proper, closed invariant subset of $G^{(0)}$ is $\{\infty\}$ and $G|_{\{\infty\}}=\{\infty\}$. On the other hand, $A(\nn\times\nn)$ is a graded ideal of $A(G)$, and hence $A(G)$ is not graded simple.
\end{example}

While we cannot obtain an analog of \cite[Theorem~2.4]{APMlocal} in general, the following theorem shows that the class of ample groupoids satisfying condition (FIC) does generalize \cite[Theorem~2.4]{APMlocal}. It is our first main theorem of this section, and more specifically,  shows that graded simplicity and graded just infinite are equivalent under mild conditions. It applies to a large class of groupoids, as we show in Corollary~\ref{DR_JI_Simple}.

\begin{theorem}
\label{thm:FIC-justinf-simple}
Let $G$ be an ample Hausdorff \'etale groupoid and let $c:G\to H$ be a continuous cocycle
into a discrete group. Assume that:
\begin{itemize}
\item $c^{-1}(e)$ is strongly effective;
\item $G$ satisfies Property~\emph{(FIC)};
\item $A(G)$ is infinite-dimensional.
\end{itemize}
Then $A(G)$ is graded just infinite (with respect to the $H$--grading induced by $c$)
if and only if $A(G)$ is graded simple.
\end{theorem}

\begin{proof}

It is straightforward that graded simple implies graded just infinite. We prove the converse. 

Assume that $A(G)$ is graded just infinite. 

We claim that $G$ is minimal. Suppose not. Then there exists a nonempty proper open
invariant subset $U\subsetneq G^{(0)}$. Let $W:=G^{(0)}\setminus U$. Then $W$ is a nonempty
proper closed invariant subset of $G^{(0)}$.

Since $c^{-1}(e)$ is strongly effective, by the graded ideal correspondence (Theorem~\ref{graded_ideals}) we have that
$A(G|_U)$ is a nonzero proper graded ideal of $A(G)$ and
\[
A(G)/A(G|_U)\ \cong\ A(G|_W).
\]
Because $A(G)$ is graded just infinite, the quotient $A(G|_W)$ must be finite-dimensional.
By Lemma~\ref{finite_dim_Steinb},
it follows that $|G|_W|<\infty$.

Now, by Property~(FIC), we deduce that $W$ is clopen. Hence $U$ is also clopen.

By Lemma~\ref{lem:clopen-decomp} we have a graded direct sum decomposition
\[
A(G)\ \cong\ A(G|_U)\ \oplus\ A(G|_W).
\]
In particular, $A(G|_W)$ is a nonzero graded ideal as well, and
\[
A(G)/A(G|_W)\ \cong\ A(G|_U).
\]
Since $A(G)$ is graded just infinite, this quotient must also be finite-dimensional. So, $\dim_K A(G|_U)<\infty$. 

Therefore both summands in the direct sum decomposition of $A(G)$ are finite-dimen\-sional, and thus
$A(G)$ is finite-dimensional, contradicting the hypothesis that $A(G)$ is
infinite-dimensional. This contradiction shows that $G$ must be minimal.

Finally, since $c^{-1}(e)$ is strongly effective, the graded ideal correspondence (Theorem~\ref{graded_ideals}) implies
that minimality of $G$ is equivalent to graded simplicity of $A(G)$.
Thus $A(G)$ is graded simple, as required.
\end{proof}

\begin{remark}
    The groupoid of Example~\ref{No_LPA_analog} does not satisfy Property (FIC) since $\{\infty\}$ is closed invariant subset with $G|_{\{\infty\}}$ finite, but $\{\infty\}$ is not open. Hence, it is not enough to only assume that $c^{-1}(e)$ is strongly effective and $A(G)$ is infinite-dimensional in Theorem~\ref{thm:FIC-justinf-simple}.
\end{remark}

The next result, however, shows that when $c^{-1}(e)$ is strongly effective and $A(G)$ is locally finite, then $A(G)$ being graded just infinite is equivalent to being graded simple. 

\begin{corollary}\label{cor:conj20-from-FIC}
Assume that $c^{-1}(e)$ is strongly effective and that $A(G)$ is locally finite and
infinite-dimensional. Then $A(G)$ is graded just infinite if and only if it is graded simple.
\end{corollary}

\begin{proof}
The result follows from Example~\ref{ex:LF.implies.FIC} and Theorem~\ref{thm:FIC-justinf-simple}.
\end{proof}

We now proceed to study when $A(G)$ is both locally finite and (graded) just infinite.

\begin{proposition}\label{LFJI_trans}
    If $A(G)$ is locally finite and just infinite, then $G$ is transitive. Moreover, $A(G)$ is isomorphic to a matrix algebra $M_{n}(A(G_{u}^{u}))$, for any $u\in G^{(0)}$ and some $n\in \mathbb{N}$. 
\end{proposition}
\begin{proof}
    Since $A(G)$ is locally finite, $|G^{(0)}|<\infty$ by Proposition~\ref{finite_fibers}. Thus, any invariant set $W\subseteq G^{(0)}$ is closed. 

    Note that if $W\subsetneq G^{(0)}$ is invariant, then $|G_v^v|<\infty$ for all $v\in W$, by Proposition~\ref{justinf_steinb}. Thus, $\dim A(G_v^v)<\infty$ whenever $v\in W$ by Lemma~\ref{finite_dim_Steinb}.

    By Corollary~\ref{direct_matrices}, $G$ has finitely many orbits, and we can express $A(G)$ as a finite direct sum 
    \begin{equation}\label{eq:direct_sum}
        A(G) = \bigoplus_{i=1}^k M_{n_i}(A(G_{u_i}^{u_i})) 
    \end{equation}
    where $k$ is the number of orbits $\mathcal{O}_i$.  Since $\dim A(G) =\infty$ by assumption, there exists a $j\in \{1,\ldots,k\}$ such that $\dim M_{n_j}(A(G_{u_j}^{u_j}))=\infty$, which implies that $|G_{u_j}^{u_j}|=\infty$. By Proposition~\ref{justinf_steinb}, $G^{(0)}=\overline{\mathcal{O}_{u_j}}=\mathcal{O}_{u_j}$. Hence, $G$ is transitive. 

    Finally, since there is only one orbit, the last statement follows from Equation~\ref{eq:direct_sum}.
\end{proof}

Though graded just infinite and graded simple are equivalent for locally finite Steinberg algebras of strongly effective groupoids, there is no hope for removing ``graded" from the statement. Locally finite and just infinite Steinberg algebras $A(G)$ are never simple. By Proposition~\ref{LFJI_trans}, such algebras have one orbit and thus all isotropy groups are isomorphic, and moreover infinite. In addition, $G$ is discrete by Remark~\ref{discrete_unitspace}. All this implies that such groupoids are always minimal as they are transitive, but $G$ is never effective as there is, up to isomorphism, one infinite and discrete isotropy group. Thus, $A(G)$ is not simple by \cite[Theorem~3.5]{Bensimple}. 

\begin{proposition}\label{transitive}
    Let $c^{-1}(e)$ be strongly effective and $A(G)$ locally finite and infinite dimensional. Then the following are equivalent. 
    \begin{enumerate}
        \item[(i)] $A(G)$ is graded just infinite. 
        \item[(ii)] $A(G)$ is graded simple.
        \item[(iii)] $G$ is transitive. 
    \end{enumerate}
    If $A(G)$ is $\mathbb{Z}$-graded, then the statements (i), (ii), and (iii) above are equivalent to 
    \begin{enumerate}
        \item[(iv)] $A(G)$ is just infinite. 
    \end{enumerate}
\end{proposition}
\begin{proof}
    Corollary~\ref{cor:conj20-from-FIC} shows that $(i)\Leftrightarrow (ii)$.  

    Proposition~\ref{LFJI_trans} shows that $(i)\Rightarrow (iii)$. 

    We show that $(iii)\Rightarrow (ii)$. Since $c^{-1}(e)$ is strongly effective, all proper graded ideals in $A(G)$ are given by proper invariant subsets in $G^{(0)}$. However, transitivity implies there are no proper invariant subsets in $G^{(0)}$. Hence, $A(G)$ is graded simple. 

    The last statement follows from \cite[Proposition 2.5]{APMlocal}
\end{proof}

\section{Applications}

\subsection{Deaconu--Renault systems}\label{DRsec}

\begin{definition}
A Deaconu--Renault system is a pair $(X,\sigma)$ where $X$ is a locally compact Hausdorff space
and $\sigma\colon \dom(\sigma)\to X$ is a local homeomorphism defined on an open subset
$\dom(\sigma)\subseteq X$.
For $n\ge 0$ set
\[
\dom(\sigma^0):=X,\qquad \sigma^0:=\id_X,\qquad 
\dom(\sigma^{n+1}):=\{x\in \dom(\sigma^n): \sigma^n(x)\in \dom(\sigma)\},
\]
and define $\sigma^{n+1}(x):=\sigma(\sigma^n(x))$ for $x\in \dom(\sigma^{n+1})$.
\end{definition}

\begin{definition}[Deaconu--Renault groupoid]
Let $(X,\sigma)$ be a Deaconu--Renault system. The \emph{Deaconu--Renault groupoid} $G(X,\sigma)$ is
\[
\begin{aligned}
G(X,\sigma):=\{(x,k,y)\in X\times\mathbb Z\times X:\ &\exists\, m,n\ge 0 \text{ with } k=m-n,\\
&x\in\dom(\sigma^m),\ y\in\dom(\sigma^n),\ \sigma^m(x)=\sigma^n(y)\}.
\end{aligned}
\]
The unit space is identified with $X$ via $x\leftrightarrow (x,0,x)$.
The range and source maps are $r(x,k,y)=x$ and $s(x,k,y)=y$, respectively, and multiplication and inverses are
given by
\[
(x,k,y)(y,\ell,z)=(x,k+\ell,z),\qquad (x,k,y)^{-1}=(y,-k,x).
\]
Basic open sets that generate the topology are of the form
\[
Z(U,m,n,V)
:=\{(x,m-n,y)\in G(X,\sigma): x\in U,\ y\in V,\ \sigma^m(x)=\sigma^n(y)\},
\]
where $m,n\ge 0$ and $U\subseteq \dom(\sigma^m)$, $V\subseteq \dom(\sigma^n)$ are open sets such that
$\sigma^m|_U$ and $\sigma^n|_V$ are homeomorphisms.
\end{definition}

\begin{remark}\label{rem:principal_unit_fibre}
The map $c\colon G(X,\sigma)\to\mathbb Z$, $c(x,k,y)=k$, is a continuous cocycle (the \emph{degree cocycle})
and induces the standard $\mathbb Z$--grading on the Steinberg algebra $A(G(X,\sigma))$. Note that $c$ restricted to each isotropy group is injective. Hence, by Proposition~\ref{prop:principal.fiber}, for Deaconu-Renault groupoids with the natural $\mathbb{Z}$-grading, $c^{-1}(0)$ is always principal, and thus strongly effective. 
\end{remark}

\noindent\textbf{Standing assumption.} From now on $X$ is also assumed to be totally disconnected. Henceforth, $G(X,\sigma)$ is an ample groupoid.

\medskip

As Example~\ref{matrix.example} shows, an analogue of \cite[Theorem~1.4]{APMlocal} does not hold for arbitrary Steinberg algebras. We now give an example showing that it may fail even in the setting of Deaconu--Renault groupoids.

\begin{example}
    Let $X=\nn$ with the discrete topology. Define $\sigma:\dom(\sigma)\to X$ by $\dom(\sigma):=\{0\}$ and $\sigma(0)=0$. Then we have that $c^{-1}(0)=G(X,\sigma)^{(0)}$ and, for every $n\in\zn\setminus\{0\}$, $c^{-1}(n)=\{(0,n,0)\}$. Hence, by Proposition~\ref{finite_fibers}, $A(G(X,\sigma))_0$ is infinite dimensional and $A(G(X,\sigma))_n$ has dimension 1 for all $n\in\zn\setminus\{0\}$.
\end{example}

Nevertheless, we obtain two results that are analogs of \cite[Theorem~1.4]{APMlocal} in the setting of Steinberg algebras of ample Deaconu--Renault groupoids. This first gives a condition that insures if one homogeneous component is infinite dimensional then all homogeneous components are infinite dimensional. The second  gives necessary and sufficient conditions for all homogeneous components to be infinite dimensional. With the second result we can recover \cite[Theorem~1.4]{APMlocal} (see Theorem~\ref{LPA_Thm1.4}).

\begin{proposition}\label{ex:DR-groupoid}
    Let $G(X,\sigma)$ be an ample Deaconu-Renault groupoid with $\sigma$ surjective, $\dom \sigma = X$, and with the standard $\mathbb Z$-grading.  If $A(G(X,\sigma))_h$ is infinite-dimensional for some $h \in \operatorname{im}(c)$,
    then $A(G(X,\sigma))_k$ is infinite-dimensional for every $k \in \operatorname{im}(c)$. 
\end{proposition}
\begin{proof}
We claim that $G(X,\sigma)$ satisfies the hypothesis of
Proposition~\ref{prop:inf-dim-transfer-finite-family} from which the conclusion immediately follows.

Fix $n\ge 0$. For each $x\in X$, since $\sigma^n$ is defined on all of $X$, we have $
(x,n,\sigma^n(x))\in G(X,\sigma),$
and hence $x\in r(c^{-1}(n))$. Thus $r(c^{-1}(n))=X$.
Also, since $\sigma$ is surjective, so is $\sigma^n$, and therefore for every $x\in X$
there exists $y\in X$ such that $\sigma^n(y)=x$. Then $
(y,n,x)\in G(X,\sigma),$
so $x\in s(c^{-1}(n))$. Hence $s(c^{-1}(n))=X$.

For $n<0$, write $n=-m$ with $m>0$. Since inversion interchanges source and range and sends
$c^{-1}(m)$ onto $c^{-1}(-m)$, it follows from the case $m>0$ that $
s(c^{-1}(n))=X=r(c^{-1}(n)).$

It follows from Proposition~\ref{prop:inf-dim-transfer-finite-family} that if one homogeneous component
$A(G(X,\sigma))_n$ is infinite-dimensional for some $n\in\mathbb Z$,
then every homogeneous component $A(G(X,\sigma))_m$ is infinite-dimensional for every $m\in\mathbb Z$.
\end{proof}

\begin{proposition}\label{prop:all.fibers.infinite.dimensional}
Let $(X,\sigma)$ be a Deaconu-Renault system. Then $A(G(X,\sigma))_k$ is infinite-dimensional for every $k \in \zn$ if and only if for every $n\in\nn$ such that $|\dom\sigma^n|<\infty$ there exist $x\in \dom\sigma^{n}$ and $k\in\nn$ such that $|\sigma^{-k}(\sigma^{n+k}(x))|=\infty$.
\end{proposition}

\begin{proof}
    We begin by noting that if $m\leq n$, then $\dom\sigma^n\scj\dom\sigma^m$.

    $(\Rightarrow)$ By Proposition~\ref{finite_fibers}, $|c^{-1}(n)|=\infty$ for all $n\in\zn$. Let $n\in\nn$ be such that $|\dom\sigma^n|<\infty$. Because $|c^{-1}(n)|=\infty$, there exists $x\in X$ such that
    \[|\{y\in X:(x,n,y)\in G(X,\sigma)\}|=\infty.\]
    Note that $x\in\dom\sigma^{n}$. Because $\operatorname{im}(\sigma^{n})$ is finite, there exists $k\in\nn$ such that
    \[|\{y\in X:y\in\dom\sigma^k,\ \sigma^{n+k}(x)=\sigma^k(y)\}|=\infty.\]
    Hence $|\sigma^{-k}(\sigma^{n+k}(x))|=\infty$.

\medskip

    $(\Leftarrow)$ The hypothesis for the converse implies that $X$ is infinite and hence $c^{-1}(0)$ is infinite. Because $c^{-1}(-n)=(c^{-1}(n))^{-1}$ for every $n\in\nn$, it is enough to prove that $c^{-1}(n)$ is infinite for all $n\in\nn\setminus\{0\}$. Note that if $\dom\sigma^n$ is infinite, then $\{(x,n,\sigma^n(x)):x\in \dom\sigma^n\}$ is infinite and contained in $c^{-1}(n)$. Otherwise, let $x$ and $k$ be such that $|\sigma^{-k}(\sigma^{n+k}(x))|=\infty$. Then the set $\{(x,n,y):\sigma^{n+k}(x)=\sigma^k(y)\}$ is infinite and contained in $c^{-1}(n)$.
\end{proof}

\begin{remark}
Proposition~\ref{ex:DR-groupoid} applies, in particular, to graph groupoids arising from finite directed graphs
with no sinks and no sources. Indeed, if $E$ is such a graph and $X=\partial E$ is its boundary path space,
then the shift map
is a surjective local homeomorphism. The assumption that $E$ has no sinks ensures that $\sigma$ is defined
on all of $\partial E$, while the assumption that $E$ has no sources ensures that $\sigma$ is surjective.
Moreover, since $E$ is finite, $\partial E$ is compact and totally disconnected. Thus $(\partial E,\sigma)$
is a Deaconu--Renault system of the type considered in Proposition~\ref{ex:DR-groupoid}, and the associated
Deaconu--Renault groupoid is canonically isomorphic to the boundary path groupoid $G_E$.

\end{remark}

Below we recover \cite[Theorem~1.4]{APMlocal}, with an alternative proof.

\begin{theorem}\label{LPA_Thm1.4}\cite[Theorem~1.4]{APMlocal}
    Let $E$ be a finite graph. The following are equivalent:
    \begin{enumerate}
        \item $L_K(E)_n$ is infinite dimensional for some $n\in\zn$.
        \item $L_K(E)_n$ is infinite dimensional for every $n\in\zn$.
        \item There exists a cycle in $E$ with an exit.
    \end{enumerate}
\end{theorem}

\begin{proof}
    (2) $\Rightarrow$ (1) is trivial.

    (3) $\Rightarrow$ (2) Let $c=c_1\ldots c_m$ be a cycle based at vertex $v$ from where an exit exists. Let $e\in s^{-1}(v)\setminus \{c_1\}$ be an exit and $\alpha\in\partial E$ be such that $s(\alpha)=r(e)$. Then for a given $n\in\nn$, we obtain the following infinite set in the domain of $\sigma^n$: $\{c^{n+l}e\alpha:l\in\nn\}$. By Proposition~\ref{prop:all.fibers.infinite.dimensional}, $L_K(E)_k$ is infinite dimensional for every $k\in\zn$.

    (1) $\Rightarrow$ (3) We prove this implication via the contrapositive. In the case that there is no cycle in $E$ with an exit, the finiteness of  $E$ implies that $\partial E$ is finite. Given $n\in\zn$, since $c^{-1}(n)\scj \partial E\times \{n\} \times \partial E$, we have that $c^{-1}(n)$ is finite. By Proposition~\ref{finite_fibers}, $L_K(E)_n$ is finite dimensional.
\end{proof}

Graded local finiteness does not imply Noetherianity for Steinberg algebras in general.
Indeed, let $X$ be a nonempty finite set and let
\[
H=\bigoplus_{n\in \mathbb N}\mathbb Z.
\]
Set $G:=X\times H$, with composition defined by
\[
(x,h)(y,k) \text{ is defined if and only if } x=y,
\quad\text{in which case}\quad
(x,h)(x,k)=(x,hk).
\]
Thus $G^{(0)}$ is identified with $X$, and each isotropy group is isomorphic to $H$.
Let $c:G\to H$ be the projection onto the second coordinate. Then, for each $h\in H$, $
c^{-1}(h)=X\times\{h\},$
so $c^{-1}(h)$ is finite. Hence $A(G)$ is graded locally finite by
Theorem~\ref{Loc_finite_Steinb}.

On the other hand, since $X$ is finite, we have $
A(G)\cong \bigoplus_{x\in X} KH.$
Now $
KH \cong K[t_1^{\pm1},t_2^{\pm1},\dots],$
which is not Noetherian, since the ascending chain of ideals
\[
(t_1-1)\subsetneq (t_1-1,t_2-1)\subsetneq (t_1-1,t_2-1,t_3-1)\subsetneq \cdots
\]
does not stabilize. Therefore $A(G)$ is not Noetherian.

Thus graded local finiteness does not imply Noetherianity for Steinberg algebras in general.
However, for Deaconu--Renault groupoids, Noetherianity is equivalent to graded local finiteness.

\begin{proposition}\label{prop:locally-finite-DR}
   Let $(X,\sigma)$ be a Deaconu-Renault system. The following are equivalent:
    \begin{enumerate}
        \item $X$ is finite,
        \item $A(G(X,\sigma))$ is locally finite,
        \item $A(G(X,\sigma))$ is Noetherian.
    \end{enumerate}
\end{proposition}

\begin{proof}
(1) $\Rightarrow$ (2) For each $n\in\zn$, there exist at most $|X|^2$ elements in $c^{-1}(n)$, so $A(G(X,\sigma))$ is locally finite by Theorem~\ref{Loc_finite_Steinb}.

(2) $\Rightarrow$ (1) Since $X\scj c^{-1}(0)$, we have that $X$ is finite by Theorem~\ref{Loc_finite_Steinb}.

(3) $\Rightarrow$ (1) It is immediate from \cite[Theorem~1.1]{Steinberg_ChainConditions_2018}.

(1) $\Rightarrow$ (3) For $G(X,\sigma)$, each isotropy group is either trivial or isomorphic to $\zn$, so its group algebra is Noetherian. The result then follows from \cite[Theorem~1.1]{Steinberg_ChainConditions_2018}.
\end{proof}

\begin{proposition}\label{prop:DR.finite}
    Let $(X,\sigma)$ be a Deaconu-Renault system. The following are equivalent:
    \begin{enumerate}
        \item $X$ is finite and $\dom(\sigma^{|X|})=\emptyset$,
        \item $G(X,\sigma)$ is finite,
        \item $A(G(X,\sigma))$ is Artinian,
        \item $A(G(X,\sigma))$ is semisimple.
    \end{enumerate}
\end{proposition}

\begin{proof}
    (2) $\Leftrightarrow$ (3) and (4) $\Rightarrow$ (2) follow immediately from \cite[Theorem~1.1]{Steinberg_ChainConditions_2018}. Assuming that $G(X,\sigma)$ is finite, each isotropy group is trivial, so (2) $\Rightarrow$ (4) again follows from \cite[Theorem~1.1]{Steinberg_ChainConditions_2018}.

    (1) $\Rightarrow$ (2) Under the assumptions of (1), we have that
    \[G(X,\sigma)\scj X\times\{n\in\zn:|n|\leq |X|\}\times X\] and the latter is finite.

    (2) $\Rightarrow$ (1) Assume that there exists $x\in \dom(\sigma^{|X|})$. By the pigeonhole principle, the set $\{x,\sigma(x),\ldots,\sigma^{|X|}(x)\}$ contains repeated elements. Thus, there exists $y\in X$ and $0< n\leq |X|$ such that $\sigma^n(y)=y$. This implies that the isotropy group of $y$ is isomorphic to $\zn$, contradicting the finiteness of $G(X,\sigma)$.
\end{proof}

With little extra work we can also characterize more general chain conditions for $A(G(X,\sigma))$. For that, we use the results of \cite{Philip2026Categorical} for Steinberg algebras.

\begin{proposition}\label{prop:cat.noether}
    Let $(X,\sigma)$ be a Deaconu-Renault system. The following are equivalent:
    \begin{enumerate}
        \item $X$ is discrete,
        \item $G(X,\sigma)$ is discrete,
        \item $A(G(X,\sigma))$ is categorically Noetherian,
        \item $A(G(X,\sigma))$ is locally Noetherian.
    \end{enumerate} 
\end{proposition}

    \begin{proof}
        Since, $G(X,\sigma)^{(0)}$ is homeomorphic to $\osf$, if $\osf$ is discrete then $G(X,\sigma)^{(0)}$ is discrete, which in turn implies that $G(X,\sigma)$ is discrete (see Remark~\ref{discrete_unitspace}). 

        Conversely, if $G(X,\sigma)$ is discrete, then so is $G(X,\sigma)^{(0)}\cong \osf$.

        The equivalence between (2), (3) and (4) follows from \cite[Theorem~5.1]{Philip2026Categorical}.
    \end{proof}

\begin{proposition}
    Let $(X,\sigma)$ be a Deaconu-Renault system. The following are equivalent:
    \begin{enumerate}
        \item $X$ is discrete and $\sigma$ has no eventually periodic points,
        \item $A(G(X,\sigma))$ is categorically Artinian,
        \item $A(G(X,\sigma))$ is locally Artinian.
    \end{enumerate}
\end{proposition}

\begin{proof}
    The equivalence between $X$ being discrete and $G(X,\sigma)$ being discrete is given by Proposition~\ref{prop:cat.noether}.

    Observe that the isotropy groups of $G(X,\sigma)$ are either trivial or isomorphic to $\zn$. Moreover, $G(X,\sigma)$ has a nontrivial isotropy group if and only if $\sigma$ admits an eventually periodic point. Since $K\zn$ is not Artinian, the result follows from \cite[Theorem~5.1]{Philip2026Categorical}.
\end{proof}

 We now proceed to characterizing graded just infinite Deaconu-Renault groupoids. We need the following result.
 
\begin{proposition}\label{prop:restriction-DR}
Let $(X,\sigma)$ be a Deaconu--Renault system, and let $W\subseteq X$ be a closed subset.
Then the following are equivalent:
\begin{enumerate}
\item $W$ is $G(X,\sigma)$-invariant in $G(X,\sigma)^{(0)}$;
\item $\sigma(W)\subseteq W$ and $\sigma^{-1}(W)\subseteq W$ (i.e. $W$  is $(\sigma,\sigma^{-1})$-invariant).
\end{enumerate}
In this case, the restricted system $(W,\sigma|_W)$ is again a Deaconu--Renault system, and
$
G(X,\sigma)|_W \cong G(W,\sigma|_W)$ 
as topological groupoids.
\end{proposition}

\begin{proof}
Assume first that $W$ is $G(X,\sigma)$-invariant.
If $x\in W\cap \dom(\sigma)$, then $(x,1,\sigma(x))\in G(X,\sigma)$, so $\sigma(x)\in W$. 
Thus $\sigma(W)\subseteq W$.
If $x\in \sigma^{-1}(W)$, then $x\in \dom(\sigma)$ and $\sigma(x)\in W$.
Since $(\sigma(x),-1,x)\in G(X,\sigma)$ and $W$ is invariant, we obtain $x\in W$.
Hence $\sigma^{-1}(W)\subseteq W$, so $W$ is $(\sigma,\sigma^{-1})$-invariant.

Conversely, assume that $W$ is $(\sigma,\sigma^{-1})$-invariant.
Let $(x,k,y)\in G(X,\sigma)$ with $x\in W$. Then there exist $m,n\ge 0$ such that $
k=m-n$ and $\sigma^m(x)=\sigma^n(y).$
Since $W$ is $\sigma$-invariant and $x\in W$, we have $\sigma^m(x)\in W$, and hence
$\sigma^n(y)\in W$.
Because $W$ is $\sigma^{-1}$-invariant, applying this property repeatedly yields $y\in W$.
Thus $W$ is $G(X,\sigma)$-invariant.

Now, since $W$ is $(\sigma,\sigma^{-1})$-invariant, Proposition~2.7 of
\cite{GoncalvesRoyerTascaNonwandering} implies that $(W,\sigma|_W)$ is again a
Deaconu--Renault system. Finally, by definition,
\[
G(W,\sigma|_W)=\{(x,k,y)\in G(X,\sigma): x,y\in W\}=G(X,\sigma)|_W,
\]
and the groupoid operations and topology agree.
\end{proof}

\begin{theorem}\label{thm:graded-justinf-DR}
Let $(X,\sigma)$ be a Deaconu--Renault system such  that $A(G(X,\sigma))$ is infinite-dimensional.
Then the following are equivalent:
\begin{enumerate}
\item $A(G(X,\sigma))$ is graded just infinite;
\item for every proper nonempty closed $(\sigma,\sigma^{-1})$-invariant subset $W\subsetneq X$,
the space $W$ is finite and
$
\dom\bigl((\sigma|_W)^{|W|}\bigr)=\emptyset.$
\end{enumerate}
\end{theorem}

\begin{proof}
Since $c^{-1}(0)$ is strongly effective by Proposition~\ref{prop:principal.fiber}  and since 
$A(G(X,\sigma))$ is infinite-dimensional, Theorem~\ref{graded_justinf_steinb}  implies that $A(G(X,\sigma))$ is graded just infinite
if and only if $
|G(X,\sigma)|_W|<\infty$ 
for every proper nonempty closed $G(X,\sigma)$-invariant subset $W\subsetneq X$.

By Proposition~\ref{prop:restriction-DR}, a closed subset $W\subseteq X$ is
$G(X,\sigma)$-invariant if and only if it is $(\sigma,\sigma^{-1})$-invariant, and in that case $
G(X,\sigma)|_W \cong G(W,\sigma|_W).$ 
Hence, for every proper nonempty closed $(\sigma,\sigma^{-1})$-invariant subset $W\subsetneq X$,
we have
\[
|G(X,\sigma)|_W|<\infty
\quad \text{if, and only if,} \quad
|G(W,\sigma|_W)|<\infty.
\]

Applying Proposition~\ref{prop:DR.finite},
we obtain
\[
|G(W,\sigma|_W)|<\infty
\quad \text{if, and only if, }\quad
W \text{ is finite and } \dom\bigl((\sigma|_W)^{|W|}\bigr)=\emptyset.
\]
Combining these equivalences for all proper nonempty closed invariant subsets $W$ gives the result.
\end{proof}

Next, we prove that under some extra assumptions on the Deaconu-Renault system, the groupoid $G(X,\sigma)$ satisfies property (FIC) and with that we obtain the equivalence of between being graded just infinite and graded simple.

\begin{proposition}\label{prop:DR.FIC}
    Let $(X,\sigma)$ be a Deaconu-Renault system and assume that the points of $X\setminus\dom(\sigma)$ are all isolated (for instance, if $\dom(\sigma)=X$). Then $G(X,\sigma)$ satisfies property (FIC).
\end{proposition}

\begin{proof}
Let $W$ be a non-empty invariant subset of $G(X,\sigma)^{(0)}$. Assume first that there exists $x\in W$ such that $x\in\dom(\sigma^n)$ for all $n\in\nn$.

For each $n\ge 0$ the defining relation of $G(X,\sigma)$ yields an element
\[
(x,n,\sigma^n(x))\in G(X,\sigma).
\]
Because $W$ is invariant and $x\in W$, it follows that $\sigma^n(x)\in W$ for all $n\ge 0$.
Hence
\[
(x,n,\sigma^n(x))\in G(X,\sigma)|_W
\qquad\text{for all }n\ge 0.
\]
If the points $\sigma^n(x)$ are pairwise distinct, then these elements are pairwise distinct,
so $|G(X,\sigma)|_W|=\infty$.

Otherwise, there exist $m\ge 0$ and $p\ge 1$ such that $\sigma^{m+p}(x)=\sigma^m(x)$.
Setting $y:=\sigma^m(x)\in W$, we obtain $\sigma^p(y)=y$, and consequently
\[
(y,\ell p,y)\in G(X,\sigma)
\quad\text{for all }\ell\in\mathbb Z,
\]
yielding infinitely many isotropy elements in $G(X,\sigma)|_W$.
Thus, again $|G(X,\sigma)|_W|=\infty$.

Assume now that $|G(X,\sigma)|_W|<\infty$. Given $x\in W$, there exists $n\in\nn$ such that $\sigma^n(x)\notin\dom(\sigma)$. By hypothesis, $\sigma^n(x)$ is an isolated point. And because $\sigma$ is a local homeomorphism, $x$ is also an isolated point. This means that $W$ consists of a finite set of isolated points, and hence it is clopen.
\end{proof}

\begin{theorem}\label{DR_JI_Simple}
    Let $(X,\sigma)$ be a Deaconu-Renault system. Assume that $A(G(X,\sigma))$ is infinite-dimensional and the points of $X\setminus\dom(\sigma)$ are all isolated. Consider $A(G(X,\sigma))$ with the natural $\mathbb{Z}$-grading. Then the following are equivalent:
    \begin{enumerate}
        \item $A(G(X,\sigma))$ is graded just infinite,
        \item $A(G(X,\sigma))$ is graded simple,
        \item $G(X,\sigma)$ is minimal.
    \end{enumerate}
\end{theorem} 

\begin{proof}
    First note that by Proposition~\ref{prop:DR.FIC}, $G(X,\sigma)$ satisfies (FIC). Thus, the equivalence between (1) and (2) follows from Theorem~\ref{thm:FIC-justinf-simple}. 
    
    If $A(G)$ is graded just infinite, then $G$ is minimal, by Proposition~\ref{justinf_steinb}. Hence, (1) implies (3). 

     Note that, by Proposition~\ref{prop:principal.fiber}, $c^{-1}(0)$ is a principal groupoid. Thus, if $G$ is minimal, then $A(G)$ has no proper graded ideals, by Theorem~\ref{graded_ideals}, and hence $A(G)$ is graded simple. Hence, (3) implies (2).
    
\end{proof}

The result below recovers \cite[Theorem~2.4]{APMlocal} and \cite[Theorem~2.7]{APMlocal}. See \cite[Example~8.4.7]{SSW_2017} or \cite{oegraph} for the definition of the graph groupoid $G_E$ for a directed graph $E$. Note that $G_E$ is a Deaconu-Renault groupoid of the shift map $\sigma$ defined on the boundary path space $\partial E$, so Theorem~\ref{DR_JI_Simple} applies.

\begin{corollary}
\label{cor:LPA-from-FIC}
Let $E$ be a row-finite directed graph and assume that $L_K(E)$ is infinite-dimensional. Then the following are equivalent
\begin{enumerate}
    \item $L_K(E)$ is graded just infinite,
    \item $L_K(E)$ is graded simple,
    \item $E$ is cofinal.
\end{enumerate}
Moreover, if $L_K(E)$ is locally finite then the conditions above are equivalent to just infiniteness of $L_K(E)$.
\end{corollary}
\begin{proof}
Since $E$ is row-finite the elements of $\partial E$ which are not in the domain of the shift map $\sigma$ are the vertices which are sinks, and they are isolated in $\partial E$. Moreover, by \cite[Proposition~8.3]{NO_19}, $G_E$ is minimal if and only if $E$ is cofinal. We can then apply Theorem~\ref{DR_JI_Simple} with the standard graded isomorphism $L_K(E)\cong A(G_E)$ to obtain the equivalence between (1), (2) and (3). The last part follows from Proposition~\ref{transitive}.
\end{proof}

\subsection{Subshift algebras}

We briefly recall definitions pertinent to subshifts and their associated groupoids, and establish some notation; see
\cite{BCGWSubshift, BCGWCStarSubshift} for further details.

Let $\alf$ be a finite or countable alphabet, and let $\osf\subseteq \alf^{\nn}$ be a one-sided
subshift with shift map $
\sigma\colon \osf\to \osf$ given by $\sigma(x_0x_1x_2\cdots)=x_1x_2x_3\cdots .
$
For $n\in\N$, let $\CL^n(\osf)$ denote the set of words of length $n$ appearing in points of $\osf$,
and set $
\lang=\bigcup_{n\ge 0}\CL^n(\osf),
$
where $\CL^0(\osf)=\{\omega\}$ and $\omega$ denotes the empty word.
For $\alpha\in \lang$, we write
$
Z(\alpha)=\{x\in \osf : x_0x_1\cdots x_{|\alpha|-1}=\alpha\}
$
for the corresponding cylinder set, and
$
F_\alpha=\{y\in \osf : \alpha y\in \osf\}
$
for the follower set of $\alpha$. More generally, for $\alpha,\beta\in\lang$, we define $C(\alpha,\beta)=\{\beta x\in \osf:\alpha x\in\osf\}$.

Let $\TCB$ be the Boolean algebra of subsets of $\osf$ generated by the set
$\{C(\alpha,\beta):\alpha,\beta\in\lang\}$, and let $\HTCB$ be the Stone spectrum of $\TCB$.
As described in \cite[Section~6]{BCGWSubshift}, the shift induces a partial local homeomorphism $
\hsig\colon \dom(\hsig)\to \HTCB,$ 
and hence a Deaconu--Renault groupoid $G(\HTCB,\hsig)$ with canonical cocycle
$c\colon G(\HTCB,\hsig)\to \mathbb Z$, $c(\xi,k,\eta)=k$.
We write
\[
\ualgshift:=A(G(\HTCB,\hsig))
\]
for the associated unital subshift algebra. Note that in \cite{BCGWSubshift}, the algebra $\ualgshift$ is defined as a universal algebra on a set of generators,  but it is shown to be isomorphic to the Steinberg algebra $A(G(\HTCB,\hsig))$ in \cite[Theorem~6.5]{BCGWSubshift}. For subshifts of finite type over finite alphabets, one has $\HTCB=\osf$ and $\hsig=\sigma$.

Recall that the topology on $\HTCB$ is given by the family $\{O_A\}_{A\in \TCB}$, where $O_A=\{\xi\in \HTCB:A\in \xi\}$. The domain of $\hsig$ is $\bigcup_{a\in\alf}O_{Z_a}$. We recall that for each $a\in \alf$, there exists a bijection $\varphih_a:O_{F_a}\to O_{Z_a}$ such that for every $\eta\in O_{Z_a}$, we have that $\hsig(\eta)=\varphih_a^{-1}(\eta)$ (see \cite[Section~6]{BCGWSubshift}). Moreover, in the finite alphabet case, $\osf=\bigcup_{a\in\alf}Z_a$ as a finite union and hence $\dom(\hsig)=\HTCB$.

We start by proving that the only finite-dimensional subshift algebra is the trivial one.

\begin{proposition}\label{No_finite_dim_subshifts}
There are no nontrivial finite-dimensional subshift algebras.  
\end{proposition}
\begin{proof}
Let $x = a_1 a_2 \ldots$ be an element of $\osf$. The algebra  $\ualgshift$ is isomorphic to a partial skew-group ring by the free group generated by $\alf$ (\cite[Theorem~5.21]{BCGWSubshift}), and as such is graded by this free group. Observe that for each $n \in \N$, there exists a nonzero element in the homogeneous component corresponding to the word $a_1 \ldots a_n$ of the free group. Since there are infinitely many such words, there are infinitely many nonzero homogeneous components and hence the algebra $\ualgshift$ must be infinite-dimensional.
\end{proof}

The following theorem is an analogue of \cite[Theorem~1.4]{APMlocal}. In particular, for infinite-dimensional subshift algebras we see that there is a dichotomy between the algebra being locally finite and all fibres being infinite-dimensional

\begin{theorem}\label{thm:all-components-infinite-iff-X-infinite}
Let $\osf$ be a one-sided subshift. Then the following are equivalent:
\begin{enumerate}
\item $\ualgshift_m$ is infinite-dimensional for every $m\in \mathbb Z$;
\item $\ualgshift_m$ is infinite-dimensional for some $m\in \mathbb Z$;
\item $\osf$ is infinite.
\end{enumerate}
\end{theorem}

\begin{proof}
(1) $\Rightarrow$ (2) is trivial.

(2) $\Rightarrow$ (3) We prove the contrapositive. Assume that $\osf$ is finite.
Since $\TCB\subseteq \mathcal P(\osf)$, we have that $\TCB$ is finite. Hence its Stone spectrum
$\HTCB$ is finite. It follows that for every $m\in\mathbb Z$, $
c^{-1}(m)\subseteq \HTCB\times\{m\}\times \HTCB$
is finite. By Proposition~\ref{finite_fibers}, each $\ualgshift_m$ is finite-dimensional.

(3) $\Rightarrow$ (1) Assume that $\osf$ is infinite. For each $x\in \osf$, define
\[
\xi_x:=\{A\in \TCB : x\in A\}.
\]
By \cite[Proposition~5.17]{BCGWSubshift}, the map
\[
\iota_\osf\colon \osf\to \HTCB,\qquad x\mapsto \xi_x
\]
is injective and satisfies $\hsig(\xi_x)=\xi_{\sigma(x)}$.
Thus $\iota_\osf(\osf)\subseteq \HTCB$ is infinite, and $\xi_x\in\dom(\hsig)$ for all
$x\in\osf$. Hence $\iota_\osf(\osf)\subseteq \dom(\widehat\sigma^m)$ for every $m\in\nn$.
By Proposition~\ref{prop:all.fibers.infinite.dimensional}, $\ualgshift_m$ is
infinite-dimensional for every $m\in\mathbb Z$.

\end{proof}

We consider the condition of a cycle having an exit in the case of subshifts, and we shall see that it is not equivalent to every fibre being infinite dimensional as it is in the graph case. Below we recall from \cite{MR4721165} the definition of a cycle. 

\begin{definition}\label{cyclexit}
Let $\osf$ be a subshift, $\alpha\in \lang\setminus\{\omega\}$, and $\emptyset\neq A\in\TCB$.
We say that $(A,\alpha)$ is a \emph{cycle} if
$A \subseteq r(A,\alpha).$ A cycle $(A,\alpha)$ is said to have an \emph{exit}
if $A\neq\{\alpha^\infty\}$. Otherwise, we say that $(A,\alpha)$ is a
\emph{cycle without exit}.
\end{definition}

\begin{remark}\label{jacare} It is observed in \cite{MR4721165} that if $(A,\alpha)$ is a cycle, then either
$A=\{\alpha^\infty\}$ or $A$ is infinite. Moreover, by \cite[Lemma~3.7]{MR4721165}, $\alpha^\infty\in \osf$.
\end{remark}

We have the following corollary to Theorem~\ref{thm:all-components-infinite-iff-X-infinite}.

\begin{corollary}\label{cor:cycle.implies.infinite}
    If $\osf$ admits a cycle with an exit, then $\ualgshift_m$ is infinite-dimensional for all $m\in\zn$.
\end{corollary}
\begin{proof}
Let $(A,\alpha)$ be a cycle with an exit. By Definition~\ref{cyclexit} and Remark~\ref{jacare}, $A$ is infinite. Since $A\subseteq \osf$, we have that $\osf$ is infinite. The result follows from Theorem~\ref{thm:all-components-infinite-iff-X-infinite}.

\end{proof}

\begin{remark}\label{Inf_implies_nocycles_false}

The converse of Corollary~\ref{cor:cycle.implies.infinite}
fails, even for subshifts over finite alphabets. Indeed, if $\osf$ is an infinite
aperiodic subshift, then, by Remark~\ref{jacare}, $\osf$ has no cycles in the sense of
Definition~\ref{cyclexit}. A concrete family of examples is given by one-sided Sturmian subshifts, see \cite{Brix2023Sturmian}.

\end{remark}

\begin{theorem}\label{capivara}
    Let $\osf$ be a nonempty subshift. The following are equivalent:
    \begin{enumerate}
        \item the subshift algebra $\ualgshift$ is locally finite,
        \item there exists $m\in\zn$ such that $\ualgshift_m$ is finite dimensional,
        \item $\osf$ is finite.
    \end{enumerate}
    In particular $\mathcal L^1(\osf)$ is finite.
\end{theorem}

\begin{proof}
    By Proposition~\ref{No_finite_dim_subshifts}, $\ualgshift$ is always infinite-dimensional. The result then follows immediately from Theorem~\ref{thm:all-components-infinite-iff-X-infinite}.
\end{proof}

We now focus on the properties of graded simplicity and graded just infiniteness. First, we recall a definition from \cite{MishaRuy}.

\begin{definition}\cite[Definition~13.20]{MishaRuy}
    Given $x\in\osf$ and $B$ a finite subset of $\lang$, the \emph{cost} of reaching $x$ from $B$ is
    \[\operatorname{Cost}(B,x)=\inf\{|\alpha|+|\gamma|:\alpha,\gamma\in\lang\text{ and }x\in C(\beta\gamma,\alpha)\text{ for all }\beta\in B\},\]
    with the convention that $\inf\emptyset=\infty$. We say that $\osf$ is \emph{hyper cofinal} if $\sup_{x\in\osf}\operatorname{Cost}(B,x)<\infty$ for all finite $B\scj\lang$ such that $F_B\neq\emptyset$. 
\end{definition}

\begin{corollary}\label{cor:graded-justinf-subshift}
Let $\osf$ be a non-empty one-sided subshift over a finite alphabet with associated groupoid  $G(\widehat U,\widehat\sigma)$. Then the following are equivalent:
\begin{enumerate}
\item $\ualgshift$ is graded just infinite;
\item $G(\widehat U,\widehat\sigma)$ is minimal;
\item $\ualgshift$ is graded simple;
\item $\osf$ is hyper-cofinal.
\end{enumerate}
Moreover, if $\osf$ is finite, the above conditions are also equivalent to:
\begin{enumerate}\setcounter{enumi}{4}
    \item $\ualgshift$ is just infinite;
    \item $G(\widehat U,\widehat\sigma)$ is transitive.
\end{enumerate}
\end{corollary}

\begin{proof}
Because $\osf$ is non-empty, the algebra $\ualgshift$ is infinite dimensional by Proposition~\ref{No_finite_dim_subshifts}. Moreover, since the alphabet is finite, we have that $\dom(\widehat\sigma)=\widehat U$, so $
\widehat\sigma\colon \widehat U\to \widehat U $
is a globally defined local homeomorphism. Hence,
Theorem~\ref{DR_JI_Simple} applies, giving the equivalence of
\textnormal{(1)}, \textnormal{(2)}, and \textnormal{(3)}.
The equivalence of \textnormal{(2)} and \textnormal{(4)} follows from \cite[Proposition~13.21]{MishaRuy} and \cite[Proposition~5.14]{BoavaCastroGoncalvesVanWyk2025CCM}.

In the case that $\osf$ is finite, $\ualgshift$ is locally finite by Theorem~\ref{capivara}. The equivalence between \textnormal{(1)} and \textnormal{(5)} follows from \cite[Proposition~2.5]{APMlocal}. The equivalence between \textnormal{(6)} and the other items follows from Proposition~\ref{transitive}.
\end{proof}

\begin{example}
    Let $\osf=\{0^{\infty},10^{\infty},20^{\infty}\}$. Since $\osf$ is finite, $\ualgshift$ is locally  finite. Moreover, we can see that $\HTCB\cong X$ and $G(\widehat U,\widehat\sigma)$ is transitive. Hence, $\ualgshift$ satisfies the equivalent statements of Corollary~\ref{cor:graded-justinf-subshift}. 
\end{example}

\begin{example}
    Let $\osf=\{0^{\infty},10^{\infty},20^{\infty}, 3^\infty\}$. In this case $\ualgshift$ is locally  finite, but it does not satisfy any of the properties given by Corollary~\ref{cor:graded-justinf-subshift}.
\end{example}

If the alphabet is infinite, we can have $\ualgshift$ being graded just infinite, but not graded simple as we shall see in the next example.

\begin{example}\label{ex:chain-markov}
Let
\[
x^{(n)}:=n(n-1)\cdots 210^\infty \qquad (n\in \mathbb N_0),
\]
and set
\[
\osf:=\{x^{(n)}:n\in \mathbb N_0\}.
\]
Then $\osf$ is a countable row-finite Markov subshift over the alphabet $\mathbb N_0$.

For every $n\ge 0$, the cylinder
$
Z\bigl(n\bigr)=\{x^{(n)}\}
$
is a singleton. In particular, all singletons belong to $\TCB$. On the other hand, in this example
every nonempty set $C(\alpha,\beta)$ is either a singleton or all of $\osf$, so $\TCB$ is precisely
the Boolean algebra of finite and cofinite subsets of $\osf$. Hence its Stone spectrum is the
one-point compactification of $\osf$:
\[
\HTCB=\osf\sqcup\{\infty\},
\]
where $\infty$ denotes the nonprincipal ultrafilter of cofinite subsets of $\osf$.

The partial map $\hsig$ is defined exactly on $\osf$. More precisely,
$\dom(\hsig)=\osf$, because each point $x^{(n)}$ contains a one-letter cylinder, while
$\infty$ contains none. Moreover,
\[
\hsig(x^{(0)})=x^{(0)},\qquad \hsig(x^{(n)})=x^{(n-1)} \quad (n\ge 1),
\]
and $\hsig(\infty)$ is not defined. Thus $\osf$ is an open invariant subset of $\HTCB$, while
$\{\infty\}$ is a closed invariant subset.

We claim that the only open invariant subsets of $\HTCB$ are $\emptyset$, $\osf$, and $\HTCB$.
Indeed, all points of $\osf$ lie in the same $G(\HTCB,\hsig)$-orbit, since
\[
\hsig^n(x^{(n)})=x^{(0)}=\hsig^m(x^{(m)})
\qquad \text{for all }m,n\ge 0.
\]
Hence any open invariant subset meeting $\osf$ must contain all of $\osf$. Since $\{\infty\}$ is not
open in the one-point compactification, the claim follows.

Now $c^{-1}(0)$ is principal for the canonical cocycle
$c\colon G(\HTCB,\hsig)\to \mathbb Z$, so by the graded ideal theorem, graded ideals of
$
\ualgshift$
correspond to open invariant subsets of $\HTCB$. Therefore the only graded ideals of $\ualgshift$ are
\[
0,\qquad \algshift:=A(G(\HTCB,\hsig)|_{\osf}),\qquad \ualgshift.
\]
In particular, $\algshift$ is the unique nonzero proper graded ideal of $\ualgshift$.
Moreover,
\[
\ualgshift/\algshift \cong A(G(\HTCB,\hsig)|_{\{\infty\}})\cong K,
\]
since the reduction to $\{\infty\}$ is the trivial one-point groupoid. Thus $\ualgshift$ is not
graded simple, but every nonzero proper graded ideal has finite-codimensional quotient. It follows that $\ualgshift$ is
graded just infinite but not graded simple.
\end{example}

\section{Declarations}

\subsection*{Ethical Approval:} This declaration is not applicable.

\subsection*{Conflicts of interests/Competing interests:} We have no conflicts of interests/com\-peting interests to disclose.

\subsection*{Authors' contributions:} 
All authors contributed equally to this work. 

\subsection*{Funding:} 
Giuliano Boava was partially supported by Funda\c{c}\~ao de Amparo \`a Pesquisa e Inova\c{c}\~ao do Estado de Santa Catarina (FAPESC).

G. G. De Castro was partially supported by Conselho Nacional de Desenvolvimento Cient\'ifico e Tecnol\'ogico (CNPq) - Brazil, and Funda\c{c}\~ao de Amparo \`a Pesquisa e Inova\c{c}\~ao do Estado de Santa Catarina (FAPESC). 

D. Gon\c{c}alves was partially supported by Conselho Nacional de Desenvolvimento Cient\'i\-fico e Tecnol\'ogico (CNPq) - Brazil, and Funda\c{c}\~ao de Amparo \`a Pesquisa e Inova\c{c}\~ao do Estado de Santa Catarina (FAPESC). 

\subsection*{Data Availability and materials:} Data sharing not applicable to this article as no datasets were generated or analysed during the current study.

\bibliographystyle{abbrv}
\bibliography{ref}

\end{document}